\documentclass{amsproc}

\usepackage[textsize=tiny]{todonotes}

\usepackage{amsmath}
\usepackage{amssymb}
\usepackage{graphicx}
\usepackage{hyperref}
\DeclareGraphicsExtensions{.png,.pdf,.eps}
\usepackage[all,2cell]{xy}
\usepackage{tikz}
\usepackage{framed}

\newtheorem{theorem}{Theorem}[section]
\newtheorem{lemma}[theorem]{Lemma}
\newtheorem{corollary}[theorem]{Corollary}
\newtheorem{proposition}[theorem]{Proposition}

\theoremstyle{definition}

\newtheorem{definition}[theorem]{Definition}

\newtheorem{remark}[theorem]{Remark}
\newtheorem{assumption}[theorem]{Assumption}

\newtheorem{example}[theorem]{Example}

\def\C{{\mathbb C}}

\def\P{{\mathbb P}}

\def\Z{{\mathbb Z}}

\def\HH{\operatorname{H\hspace{0.5pt}}}

\def\cD{{\mathcal D}}
\def\cE{{\mathcal E}}

\def\cM{{\mathcal M}}
\def\cN{{\mathcal{N}}}
\def\cO{{\mathcal{O}}}

\def\cU{{\mathcal U}}

\def\cX{{\mathcal X}}
\def\cY{{\mathcal Y}}

\def\fg{{\mathfrak g}}

\def\PGL{\operatorname{PGl}}
\def\PSP{\operatorname{PSp}}

\def\rat{\dashrightarrow}

\def\operatorname#1{\mathop{\rm #1}\nolimits}

\def\DA{{\rm A}}
\def\DB{{\rm B}}
\def\DC{{\rm C}}
\def\DD{{\rm D}}

\def\DF{{\rm F}}
\def\DG{{\rm G}}

\def\Proj{\operatorname{Proj}}

\def\Hom{\operatorname{Hom}}

\def\Pic{\operatorname{Pic}}

\def\codim{\operatorname{codim}}

\def\rk{\operatorname{rk}}

\def\deg{\operatorname{deg}}

\def\rat{\operatorname{RatCurves}}

\def\NE{{\operatorname{NE}}}

\def\Nu{{\operatorname{N_1}}}
\def\NU{{\operatorname{N^1}}}

\newcommand{\cNE}[1]{\overline{\NE(#1)}}

\def\ev{{\operatorname{ev}}}

\def\ad{\operatorname{ad}}

\newcommand{\pb}{\ar@{}[dr]|{\text{\pigpenfont J}}}

\newcommand{\ol}[1]{\overline{#1}}

\makeindex

\begin{document}

\title{Manifolds with two projective bundle structures}

\author[Occhetta]{Gianluca Occhetta} \address{Dipartimento di
  Matematica, Universit\`a di Trento, via Sommarive 14 I-38123 Povo di
  Trento (TN), Italy} \thanks{First and third author supported by PRIN project
  ``Geometria delle variet\`a algebriche''. Second author supported  by MIUR project FFABR, and by the Polish National Science Center project 2013/08/A/ST1/00804}
\email{gianluca.occhetta@unitn.it, eduardo.solaconde@unitn.it}

\author[Romano]{Eleonora A. Romano}
\address{Instytut Matematyki UW, Banacha 2, 02-097 Warszawa, Poland}
\email{elrom@mimuw.edu.pl}

\author[Sol\'a Conde]{Luis E. Sol\'a Conde}


\begin{abstract}
In this paper we classify varieties of Picard number two having two projective bundle structures of any relative dimension, under the assumption that these structures are mutually uniform. As an application we prove the Campana--Peternell conjecture for varieties of Picard number one admitting $\C^*$-actions of a certain kind.
\end{abstract}

\maketitle

\section{Introduction}\label{sec:intro}

Within the class of smooth complex  projective varieties, having a projective bundle structure is an uncommon property (see \cite{AD}).  
As a consequence, one does not expect to find many varieties of low Picard number having more than one of these structures. Besides the trivial constructions (products), the standard example is the projectivization of the tangent bundle of $\P^n$, whose second contraction is a $\P^{n-1}$-bundle over $\P^{n\vee}$. Looking at this example from the point of view of Representation Theory, one may easily construct other examples of the kind among rational homogeneous varieties. Apart of them, only one example is known, and it supports a large group of automorphism (see \cite{Kan}). 
More generally, one expects the interplay among the different structures of projective bundle to be the cause of the existence of automorphisms of the variety. This feature is well understood in the case in which these structures have relative dimension one  and there are as many  as the Picard number of the variety (\cite{OSWW,OSW}); the case of projective bundle structures of higher relative dimension is, to our best knowledge, still unexplored. 

Nevertheless, even the simplest case of varieties of Picard number two having two projective bundle structures appears naturally in  different situations. In the setup of Projective Geometry, we find them in the problem of classifying smooth projective subvarieties $X\subset \P^N$ having smooth dual (see \cite{Ein1}). 
Within Birational Geometry, they appear as the exceptional divisors of simple $K$-equivalent maps (cf. \cite{Kan2}), a class containing Mukai flops. Classification results for manifolds with two projective bundle structures could be useful in the study of this type of maps.

In this paper we will consider the problem of classifying varieties of Picard number two having two projective bundle structures which are mutually uniform. This means that the pullback of one of the structures to a line in a fiber of the other  is independent of the chosen line. The tight relation of uniformity with homogeneity is well documented in the literature (see \cite{MOS5} 
and the references therein), and in our setting allows us to prove the main result of this paper (Theorem \ref{thm:main}), stating that a variety satisfying the above conditions is homogeneous. We observe here that the result is not true without the uniformity assumption: the smooth $5$-dimensional quadric admits a $\P^2$-bundle whose second contraction is a $\P^2$-bundle (cf. \cite{Kan, O}). 

Finally, in Theorem \ref{thm:drumRH}, we present an application of our result in the framework of Campana--Peternell Conjecture, which states that a smooth Fano variety with nef tangent bundle is homogeneous. Recently, Li (cf. \cite{Li}) has shown that horospherical varieties with nef tangent bundle are homogeneous; after reducing to the main case of horospherical varieties of Picard number one, the author uses the classification due to Pasquier (see \cite{Pas}), and shows that none of the non homogeneous examples has nef tangent bundle. 
Horospherical varieties of Picard number one admit a $\C^*$-action of bandwidth one (the bandwidth of an action is a measure of its complexity, see Definition \ref{def:bandwidth}). Using our main result we deduce Theorem \ref{thm:drumRH}, which states that smooth varieties of Picard number one with nef tangent bundle admitting a $\C^*$-action of bandwidth one are rational homogeneous.\par
\medskip
\noindent{\bf Outline:} We start the paper with some background material (Section \ref{sec:prelim}) on rational homogeneous varieties and bundles. We recall the definition of tag of a flag bundle on the projective line, a concept that allows us to talk about uniformity of rational homogeneous bundles. We finish the section by recalling the concept of nesting of a rational homogeneous bundle, and a result about them, 
 that we will use in our proof to detect symplectic structures on projective bundles. In Section \ref{set:main} we state our Main Theorem and present the notation we will use along its proof. We prove the Theorem in Section \ref{sec:proof_main}; starting with some preliminary results, we show in Section \ref{ssec:proof} how to construct a variety dominating our initial variety $Y$, and having as many $\P^1$-bundle structures as its Picard number. Then by  \cite[Theorem A.1]{OSW} this is a flag manifold, and the homogeneity of $Y$ follows. Finally, Section \ref{sec:drums} is devoted to the application to the Campana--Peternell conjecture mentioned above. 


\section{Preliminaries}\label{sec:prelim}

Throughout the paper we will work with complex projective varieties. We denote by
$\rho_{X}:= \dim{\NU(X)}=\dim{\Nu(X)}$ the \textit{Picard number} of $X$. Given a vector bundle $\cE$ on a variety $X$, $\P(\cE)$ denotes the Grothendieck's projectivization of $\cE$, that is $\Proj(\bigoplus_{m\geq 0}S^m\cE)$.


\subsection{Notation on rational homogeneous varieties}\label{ssec:notnRH}
We introduce some notation regarding semisimple algebraic groups and their projective quotients; we refer to \cite{FuHa, Hum2} for details on this topic. Along the paper $G$ will denote a semisimple algebraic group, $B$ a {\em Borel subgroup} $B \subset G$, and $H$ a {\em Cartan subgroup} $H\subset B$ (i.e., a maximal torus contained in $B$). The torus $H$ determines a {\em root system} $\Phi$, whose {\em Weyl group} $W$ is isomorphic to the quotient ${\rm N}(H)/H$ of the normalizer ${\rm N}(H)$ of $H$ in $G$. Within $\Phi$, $B$ determines a base of positive simple roots $\Delta=\{\alpha_i,\,\, i=1,\dots,n\}$ whose associated reflections we denote by $s_i$. Let $\cD$ be the {\em Dynkin diagram} whose set of nodes is $\Delta$. When $G$ is simple, the nodes of the Dynkin diagram will be numbered as in the standard reference \cite[p.~58]{Hum3} and we will identify each node $\alpha_i\in \Delta$ with the corresponding index $i$.  

For every nonempty subset $I\subset \Delta$, denoting by $I^c$ its complement $\Delta\setminus I$, we may consider a {\em parabolic subgroup} $P(I^c)$ defined as $P(I^c):=BW(I^c)B$, where $W(I^c)\subset W$ is the subgroup of $W$ generated by the reflections $s_i$, $i\in I^c$. Quotienting by the subgroups $P(I^c)$ we obtain the {\em rational homogeneous varieties} $G/P(I^c)$ (cf. \cite[\S~23.3]{FuHa}), so that an inclusion $I\subset J\subset \Delta$ provides a smooth contraction $G/P(J^c)\to G/P(I^c)$. Given $I\subset\Delta$ it is standard to represent $G/P(I^c)$ by the Dynkin diagram $\cD$ marked on the nodes $I$. For this reason, and in order to have a unified notation within the class of rational homogeneous varieties, we will set:
$$
\cD(I):=G/P(I^c).
$$
The variety $G/B=\cD(\Delta)$ is called the {\em complete flag manifold} associated to $G$. 
 \subsection{Generalities on flag bundles}
\begin{definition} Let $X$ be a complex smooth projective variety, and $G$ be a semisimple group with Dynkin diagram $\cD$. A {\em rational homogeneous bundle} (or $G/P$-bundle) $p:Y\to X$ is a smooth morphism whose fibers are isomorphic to $G/P$, where $P$ is a parabolic subgroup of $G$. If $P=B$ is a Borel subgroup of $G$, we  say that $p$ is a {\em flag bundle} of type $\cD$ (also called $\cD$-bundle, or $G/B$-bundle).  
\end{definition}

Throughout the paper we will be mostly interested in the case in which $X$ is rationally connected and, in particular, simply connected.  
It then follows that a $G/P$-bundle $p:Y\to X$ is completely determined by a cocycle $\theta\in \HH^1(X,G_{\ad})$, where $G_{\ad}$ is the adjoint group of the Lie algebra $\fg$ of $G$. Since $G/P$ can be seen as a quotient of $G_{\ad}$, it is harmless to assume that $G=G_{\ad}$ (see \cite[Section~2]{MOS5} for details). 
Moreover, if   $E\to X$ is the $G$-principal bundle defined by the cocycle $\theta$, the $G/P$-bundle $p:Y\to X$ will be isomorphic to:  
$$E\times^GG/P:=(E\times G/P)/\sim,\qquad (e,gP)\sim(eh,h^{-1}gP),\,\,\forall h\in G.$$ 
In a similar way, for every $I\subset \Delta$ we construct a $\cD(I)$-bundle denoted by $\pi_I:\cY(I):=E\times^G \cD(I)\to X$. Note that this notation is slightly different from the one used in \cite{MOS5}, where the $\cD(I)$-bundle $\cY(I)$ is denoted by $\cY_{\Delta\setminus I}$.

In particular, setting $\pi:=\pi_\Delta$, $\cY:=\cY(\Delta)$, we get a flag bundle $\pi\colon\cY\to X$, whose projection factors by $\pi_I$, for every $I\subset \Delta$: 
$$
\xymatrix{\cY\ar@/^1pc/[rr]^{\pi}\ar[r]_{\rho_{I^c}}&\cY(I)\ar[r]_{\pi_I}&X}
$$
Note that $\rho_{I^c}$ is a flag bundle over $\cY(I)$, with fibers isomorphic to $\cD_{I^c}(I^c)$, where $\cD_{I^c}$ is the Dynkin subdiagram of $\cD$ supported on the set of nodes ${I^c}$.

\begin{definition}\label{def:reduce}
Given a subgroup $G'\subset G$, and a flag bundle $\pi:\cY\to X$ as above, we say that $\pi$ {\em reduces to }$G'$ if the cocycle $\theta\in \HH^1(X,G)$ defining $\pi$ belongs to the image of the natural map $\HH^1(X,G')\to \HH^1(X,G)$.
\end{definition}

We will be mostly interested in the case in which $G'$ is semisimple, so that $\theta$ defines a flag bundle $\pi':\cY'\to X$, where $\cY'$ is contained in $\cY$. 

\begin{example}\label{ex:C_structure}
In the case in which $\pi$ is a flag bundle of type $\DA_r$, with $r$ odd, saying that $\pi$ reduces to $\PSP(r+1)\subset\PGL(r+1)$ is equivalent to say that the $\P^r$-bundle $\pi_1:\cY(1)\to X$ supports a relative contact structure, so that $\cY'$ is the subvariety $\cY'\subset \cY$ parametrizing flags in the $\P^r$'s, isotropic with respect to the contact structure. In this case $\cY'$ is a flag bundle of type $\DC_{(r+1)/2}$.
\end{example}

\begin{definition}\label{def:diag}
A flag bundle $\pi:\cY\to X$ is called {\em diagonalizable} if its defining cocycle $\theta\in \HH^1(X,G)$ belongs to the image of the natural map $\HH^1(X,H)\to \HH^1(X,G)$, where $H\subset G$ denotes a Cartan subgroup. 
\end{definition}

The next statement provides a geometric criterion for the diagonalizability of flag bundles (cf. \cite[Corollary 3.10]{MOS5}). 

\begin{proposition}\label{prop:sect_diag}
Let $\pi:\cY\to X$ be a flag bundle, and assume that $X$ is a Fano manifold of Picard number one. Then  $\pi$ is diagonalizable if and only if it admits a section.
\end{proposition}

\subsection{Flag bundles on $\P^1$}\label{ssec:flagP1}
The case of flag bundles on the projective line $\P^1$ is particularly simple. In fact, following \cite[Section 4]{MOS5}, any $G/B$-bundle ${\pi}\colon \cY\to \mathbb{P}^{1}$ is completely determined by the  Dynkin diagram $\cD$ of $G$, together with a map $\delta:\Delta\to \Z$, sending $i$ to  $d_i=K_i\cdot \Gamma_0$, being $K_i$ the relative canonical of the elementary contraction $\rho_i\colon \cY\to \cY(i^c):=\cY(\{i\}^c)$ with $i\in \Delta$, and $\Gamma_0$ a minimal section of ${\pi}$, i.e., a section without deformations with a point fixed. Such $\delta$ is called \textit{tag of ${\pi}$}. When an ordering on $\Delta$ is given, we will denote $\delta$ by the $n$-tuple $(d_1,\dots,d_n)$. 

\begin{example}
In the case of a $G/B$-bundle of type $\DA_r$ over $\P^1$, obtained as the complete flag bundle associated to a vector bundle $\bigoplus_{i=0}^r\cO_{\P^1}(a_i)$ with $a_0\leq \dots\leq a_r$, the tag of the $G/B$-bundle is $(a_1-a_0,\dots,a_r-a_{r-1})$, when ordering the nodes of $\DA_r$ from left to right:
$$
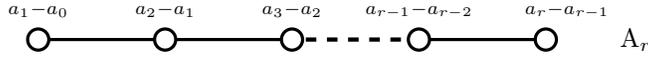
\ifx\du\undefined
  \newlength{\du}
\fi
\setlength{\du}{4\unitlength}
\begin{tikzpicture}
\pgftransformxscale{1.000000}
\pgftransformyscale{1.000000}

\definecolor{dialinecolor}{rgb}{0.000000, 0.000000, 0.000000} 
\pgfsetstrokecolor{dialinecolor}
\definecolor{dialinecolor}{rgb}{0.000000, 0.000000, 0.000000} 
\pgfsetfillcolor{dialinecolor}


\pgfsetlinewidth{0.300000\du}
\pgfsetdash{}{0pt}
\pgfsetdash{}{0pt}

\pgfpathellipse{\pgfpoint{-6\du}{0\du}}{\pgfpoint{1\du}{0\du}}{\pgfpoint{0\du}{1\du}}
\pgfusepath{stroke}
\node at (-6\du,0\du){};

\pgfpathellipse{\pgfpoint{6\du}{0\du}}{\pgfpoint{1\du}{0\du}}{\pgfpoint{0\du}{1\du}}
\pgfusepath{stroke}
\node at (6\du,0\du){};

\pgfpathellipse{\pgfpoint{18\du}{0\du}}{\pgfpoint{1\du}{0\du}}{\pgfpoint{0\du}{1\du}}
\pgfusepath{stroke}
\node at (18\du,0\du){};

\pgfpathellipse{\pgfpoint{30\du}{0\du}}{\pgfpoint{1\du}{0\du}}{\pgfpoint{0\du}{1\du}}
\pgfusepath{stroke}
\node at (30\du,0\du){};

\pgfpathellipse{\pgfpoint{42\du}{0\du}}{\pgfpoint{1\du}{0\du}}{\pgfpoint{0\du}{1\du}}
\pgfusepath{stroke}
\node at (42\du,0\du){};

\pgfsetlinewidth{0.300000\du}
\pgfsetdash{}{0pt}
\pgfsetdash{}{0pt}
\pgfsetbuttcap

{\draw (-5\du,0\du)--(5\du,0\du);}
{\draw (7\du,0\du)--(17\du,0\du);}
{\draw (31\du,0\du)--(41\du,0\du);}

\pgfsetlinewidth{0.400000\du}
\pgfsetdash{{1.000000\du}{1.000000\du}}{0\du}
\pgfsetdash{{1.000000\du}{1.000000\du}}{0\du}
\pgfsetbuttcap
{\draw (19.3\du,0\du)--(29\du,0\du);}

\node[anchor=west] at (48\du,0\du){${\rm A}_r$};

\node[anchor=south] at (-6\du,1.1\du){$\scriptstyle a_1-a_0$};

\node[anchor=south] at (6\du,1.1\du){$\scriptstyle a_2-a_1$};

\node[anchor=south] at (18\du,1.1\du){$\scriptstyle a_3-a_2$};

\node[anchor=south] at (30\du,0.95\du){$\scriptstyle a_{r-1}-a_{r-2}$};

\node[anchor=south] at (44\du,0.95\du){$\scriptstyle a_r-a_{r-1}$};

\end{tikzpicture} 
$$
\end{example}

Later on we will make use of the following two statements:

\begin{lemma}\label{lem:subtag}
Let $\pi:\cY \to \P^1$ be a $\cD$-bundle, with tag $\delta: \Delta \to \Z$. Let $s: \P^1 \to \cY(I)$ be a minimal section.
The tag of the $\cD_{I^c}$-bundle $s^*\cY:=\cY\times_{\cY(I)} \P^1$ is  $\delta|_{I^c}$.
\end{lemma}

\begin{proof}
A minimal section of $s^*\cY$ over $\P^1$ maps via the natural inclusion $\iota:s^*\cY\hookrightarrow \cY$ to a minimal section of $\cY$ over $\P^1$. The tag of $s^*\cY$ is given by the intersections of the relative canonical line bundles $K_i$, $i\in I^c$, of the elementary contractions of $s^*\cY$ over $\P^1$. Since these relative canonicals are the pullbacks via $\iota$ of the corresponding relative canonicals of $\cY$ over $\P^1$, the result follows.
\end{proof}

\begin{lemma}\label{lem:minimals}
Let $\pi: \cY\to \P^1$ and $\pi':\cY' \to \P^1$ be flag bundles, with Dynkin diagrams $\cD$ and $\cD'$, having nodes indexed by $\Delta$ and $\Delta'$, and tags $\delta$, $\delta'$.
Set $I_0:=\{i \in \Delta~|~ \delta(i)=0\}$, $I'_0:=\{i \in \Delta'~|~ \delta'(i)=0\}$, $N':=I'^c_0$, and let $J$ be a subset of $I_0$. Assume that there exists a commutative diagram
$$ \xymatrix@C=20pt@R=10pt{
\cY(J) \ar[rd]_{\pi_{J}} \ar[rr]^{s} & &\cY'(N') \ar[ld]^{\pi'_{N'}}\\
 &\P^1 &}$$
such that the image of a minimal section of $\pi_{J}$ is the minimal section of $\pi'_{N'}$. Then $s$ is relatively constant.
\end{lemma}

\begin{proof} 
It is enough to prove that the composition of $s$ with the natural projection $\rho_{J^c}:\cY\to \cY(J)$ is relatively constant over $\P^1$. Denoting by $\Gamma_k$ the fibers of the elementary contractions $\rho_k$ of $\cY$ over $\P^1$, we are left to show that $s\circ\rho_{J^c}$ contracts all the curves $\Gamma_k$, $k\in J$. 

Given $k\in J$, denoting by $\Gamma$ a minimal section of $\pi:\cY\to \P^1$, the hypothesis $\delta(k) = 0$ tells us that $\rho_k^{-1}(\rho_k(\Gamma))\simeq \Gamma\times \Gamma_k\subset \cY$, and the fibers of the projection $\Gamma\times \Gamma_k\to \Gamma_k$ are minimal sections of $\cY$ over $\P^1$. Since we are assuming that $s$ maps minimal sections of $\pi_{J}$ to the minimal section of $\pi'_{N'}$, it follows that $\rho_k^{-1}(\rho_k(\Gamma))$ gets contracted by $s\circ\rho_{J^c}$ to the unique minimal section of $\pi'_{N'}:\cY'(N')\to\P^1$. In other words, $s\circ\rho_{J^c}$ contracts $\Gamma_k$ to a point.
\end{proof}


\subsection{Uniform flag bundles}

In this section we focus on another property related to flag bundles, which is the \textit{uniformity}. We refer to \cite[Section 5]{MOS5} for further details. 
 A {\it family of rational curves} $\cM$ on $X$ is an irreducible  subvariety  $\cM \subset \rat^n(X)$ with universal family $p\colon \cU\to \cM$, and evaluation morphism $\ev\colon \cU\to X$. Note that we do not require $\cM$ to be a {\em complete family}, i.e.,  an irreducible component of $\rat^n(X)$.
A family of rational curves is called {\it dominating} if the evaluation $\ev$ is dominating, and {\it unsplit} if $\cM$ is proper. 

If $\cM$ is a family of rational curves as above, and $\pi\colon \cY\to X$ is a $G/B$-bundle, then the fiber product ${\ev}^{*}\cY=\cY\times_{X} \cU$ has a natural structure of $G/B$-bundle over $\cU$ (cf. \cite[Section 2.1]{MOS5}), which by abuse we continue to denote by $\pi$:
$$ \xymatrix@C=20pt{
{\ev}^{*}\cY\ar[d]_\pi \ar[r]^{}&{\cY}\ar[d]_{\pi}\\
\cU \ar[r]_q & X}$$

For every rational curve $C=p^{-1}(z)\subset \cU$ the pullback of the $G/B$-bundle ${\ev}^{*}\cY$ to $C$ is determined by its tag $\delta_{C}(\cY)=(d_1,\dots,d_n)$ on the Dynkin diagram of the group $G$. Hence it make sense to give the following definition:

\begin{definition} Let $X$ be a complex smooth projective variety with a dominating family $\cM$ of rational curves. We say that a $G/B$-bundle $\pi\colon \cY\to X$ is {\em 	uniform with respect to} $\cM$ if the tag $\delta_{C}(\cY)$ does not depend on the chosen curve $C\in \cM$. In this case, the tag of $\pi$ will be denoted by $\delta(\cY)$ or simply by $\delta$. Moreover, a rational homogeneous bundle $Y\to X$  is said to be uniform with respect to $\cM$ if the associated complete flag bundle is uniform.
\end{definition}

The following statement characterizes trivial flag bundles by their tags:

\begin{theorem}[{\cite[Theorem 5.5]{MOS5}}]\label{thm:trivial}
Let $X$ be a manifold which is rationally chain connected with respect to  $\cM_1, \dots,\cM_s$, unsplit families of rational curves, and $\pi:\cY\to X$ a $G/B$ bundle over $X$. 
Assume that for every rational curve $C_i\in\cM_i$ we have $\delta_{C_i}(\cY)=(0, \dots, 0)$.
Then $ \cY\cong X\times G/B$ is trivial.
\end{theorem}

Let $\pi:\cY\to X$ be a $G/B$-bundle, uniform with respect to an unsplit dominating family $\cM$ of rational curves, with tag $\delta$.
We set 
\begin{equation}\label{eq:I0}
I_0:=\left\{i\in \Delta|\,\,\delta(i)=0\right\}\subset \Delta,\qquad N:=I_0^c\subset \Delta 
\end{equation}
and denote by $F\simeq\cD_{I_0}(I_0)$ the fiber of $\rho_{I_0}\colon\cY\to \cY(N)$.  
Then over every rational curve $C$ of the family $\cM$ we have a well defined trivial proper subflag $F\times C\subset\pi^{-1}(C)$. This subflag determines a section of the restriction of the pullback $\ev^*\cY(N)\to \cU$ to $C$, for every $C$, and one may then prove (cf. \cite[Section~6.1]{MOS5}) that all these sections glue together into a section: 

\begin{lemma}\label{lem:map_sections}
Let $X$ be  smooth projective variety, dominated by an unsplit  family $\cM$ of rational curves, with evaluation  $\ev:\cU\to X$. Let $\pi\colon\cY\to X$ be a flag bundle, uniform with respect to $\cM$. Then there exists a morphism 
\begin{equation}\label{eq:section0}
s_{0}:\cU\to \cY(N), 
\end{equation}
such that $\pi_{N}\circ s_0=\ev$, sending a fiber $\ell$ of $\cU\to \cM$ to the unique minimal section of the pullback of $\pi_{N}:\cY(N)\to X$ to $\ell$.
\end{lemma}
 
The following statement is a diagonalizability criterion for uniform flag bundles, that can be read in \cite[Corollary 6.5]{MOS5}.

\begin{proposition}\label{prop:diag}
Let $X$ be a Fano manifold of Picard number one, $\pi:\cY\to X$ be a $G/B$-bundle over $X$, uniform with respect to an unsplit and dominating family $\cM$ of rational curves  in $X$, and let $I_0$, $N$ be as in (\ref{eq:I0}). Assume that the section $s_0:\cU\to \cY(N)$ factors via $q:\cU\to X$. Then $\pi$ is diagonalizable. 
\end{proposition}


\subsection{Nestings}
 
Later on we will use some results on sections of the natural projections of rational homogeneous bundles, called {\em nestings}; see \cite{MOS7} for details.

\begin{definition}
Given a Dynkin diagram $\cD$ and two disjoint nonempty sets of nodes $I,J$ of $\cD$, a {\em nesting of type} $(\cD,I,J)$ is a section of the contraction $\cD(I\cup J)\to \cD(I)$. Given a flag bundle $\cY\to X$ of type $\cD$, a {\em nesting of type} $(\cY,I,J)$ is a section of the contraction $\cY(I\cup J)\to \cY(I)$.
\end{definition}

The results obtained in \cite{MOS7} essentially say that this kind of maps are very rare. In this paper we will only make use of them in the case of bundles of type $\DA$, as summarized in the following statement:

\begin{proposition}\label{prop:nest}
Let $\pi: \cY \to X$ be a flag bundle of type  $\DA_r$ on a smooth projective variety $X$, uniform with respect to an unsplit dominating family of rational curves $\cM$, with tag $\delta=(d_1,\dots,d_r)$, and admitting a nesting of type $(\cY,I,J)$. Then $r$ is odd, $I\cup J=\{1,r\}$, and $d_i=d_{r+1-i}$ for all $i$. Moreover $\cY$ reduces to $\PSP(r+1)$, and the associated flag bundle of type $\DC_{(r+1)/2}$ is uniform with respect to $\cM$ with tag $(d_1,\dots,d_{(r+1)/2})$.
\end{proposition}

\begin{proof}
Restricting the nesting to fibers over points of $X$, and using \cite[Theorem 1.1]{MOS7}, we get that $r$ is odd and that, up to reverting the order of the nodes of $\DA_r$, $I=\{1\}$, $J=\{r\}$. By \cite[Corollary 3.9, Proposition 4.3]{MOS7}, the bundle reduces to $\PSP(r+1)$ and, applying \cite[Proposition 4.11]{MOS7}, the tag $\delta$ is symmetric. The last observation follows by the definition of tag.
\end{proof}


\section{Setup}\label{set:main}

Let $Y$ be a smooth projective variety of Picard number two which admits two projective bundle structures  $p_{\pm}:Y\to X_{\pm}$ of relative dimensions $r_\pm:=\dim Y -\dim X_\pm$. 
They define two complete flag bundles (of type $\DA_{r_+}$, $\DA_{r_-}$, respectively) over $X_+$ and $X_-$; we denote them by $\cY^+,\cY^-$:
$$
\xymatrix@R=7pt{\cY^-\ar[rd]\ar[dd]&&\cY^+\ar[ld]\ar[dd]\\&Y\ar[ld]_{p_-}\ar[rd]^{p_+}&\\X_-&&X_+
}
$$

We denote by $\cM_+$ the family of lines in the fibers of $p_+$, with universal family $\ev_+:\cU_+ \to Y$. We can think of $\cM_+$  as a dominating unsplit family of rational curves in $X_-$, which is not in general a complete family:
$$ \xymatrix@R=18pt{
\cU_+\ar[d] \ar[r]^{\ev_+}& Y \ar[r]^{p_-} \ar[d]_{p_+} & X_- \\
\cM_+ \ar[r]& X_+& }$$
The map $\ev_+\colon\cU_+ \to Y$ is a $\P^{r_+-1}$ bundle, given by the projectivization of the relative cotangent bundle of $Y \to X_+$, and the composition $p_+\circ \ev_+\colon\cU_+\to X_+$ is a partial flag bundle (of points and lines on the fibers of $p_+$), that can be obtained as a contraction of $\cY^+$. The same holds for the family $\cU_-\to \cM_-$  of lines in the fibers of $p_-$, with evaluation map $\ev_-:\cU_-\to Y$.

\begin{assumption}\label{ass:mutual}
In the sequel we will  assume that  $p_-\colon Y\to X_-$ (resp. $p_+$) is uniform with respect to the curves of $\cM_+$ (resp. $\cM_-$). We will say that $p_+$ and $p_-$ are {\em mutually uniform}. 
\end{assumption}

We will denote by $\delta_-$ the tag of the flag bundle $\cY^-\to X_-$  with respect to the family $\cM_+$, and set $I_{0-}:=\left\{j\in \Delta_-|\,\,\delta_-(j)=0\right\}\subset \Delta_-$, $N_-:=I_{0-}^c$, where $\Delta_-$ denotes the set of nodes of the Dynkin diagram  $\DA_{r_-}$. The unmarking of the nodes of $I_{0-}$ defines a factorization of $\cY^-\to X_-$ via the rational homogeneous bundle $\cY^-(N_-)\to X_-$. In the same way we define $\delta_+$, $\Delta_+$, $I_{0+}$, and $N_+$.

\begin{theorem}[Main Theorem]\label{thm:main}
Let $Y$ be a smooth projective variety with $\rho_Y=2$, supporting two mutually uniform projective bundle structures. Then $Y$ is rational homogeneous. 
\end{theorem}

\begin{remark}\label{rem:list}
With the notation introduced in Section \ref{ssec:notnRH}, one may check that the list of varieties satisfying the assumptions of the Theorem is:
$$\begin{array}{c}
\DA_n(1,n),\,\,n\geq 2;\quad\DA_n(r,r+1),\,\,n\geq 2,\,\,r<n;\quad \DB_3(1,3);\quad \DB_n(n-1,n),\,\,n\geq 2;\\[2pt]
\DC_n(r,r+1),\,\,n\geq 2,\,\, r<n;\quad \DD_n(n-1,n),\,\,n\geq 4; \quad \DF_4(2,3);\quad  \DG_2(1,2).
\end{array}
$$
\end{remark}

Since $\cM_-$ is an  unsplit dominating family of rational curves in $X_+$ with respect to which $p_+$ is uniform, by Lemma \ref{lem:map_sections}, there exists a morphism $s_{0-}:\cU_- \to \cY^+(N_+)$ over $X_+$; analogously, we obtain a morphism $s_{0+}:\cU_+ \to \cY^-(N_-)$.
The above maps fit in the following commutative diagram:
$$
\xymatrix
@R=20pt
{&\mathcal Y^-\ar[d]
\ar[r]&\cY^-(N_-)&&\cY^+(N_+)&\cY^+\ar[d]\ar[l]
&\\
\cM_-&\cU_-\ar[d]\ar[l]\ar[rr]_{\ev_-}\ar[urrr]_{s_{0-}}&&Y\ar[lld]^{p_-}\ar[rrd]_{p_+}&&\cU_+\ar[ulll]^{s_{0+}}\ar[d]\ar[r]\ar[ll]^{\ev_+}&\cM_+\\
&X_-&&&&X_+&}
$$

In the case in which $\delta_+(1)\neq 0$, the natural projection $\cY^+(N_+)\to X_+$ factors via $Y$. The next statement shows that if $\delta_+(1)= 0$ then the section $s_{0-}$ can be lifted to a bundle $\cY^+(I)$ dominating $\cY^+(N_+)$ and $Y$.

\begin{lemma}\label{lem:YI}
Assume that $\delta_+(1)=0$. Set $I:=N_+\cup\{1\}$.  
Then there exists a morphism $f:\cU_-\to \cY^+(I)$ fitting in the following commutative diagram:
\begin{equation}\label{diag:Y_I}
\xymatrix@R=10pt{&&\cY^+(N_+)\ar[rd]&\\\cU_-\ar@/^/[rru]^{s_{0-}}\ar@/_/[rrd]_{\ev_-}\ar[r]^{f}&\cY^+(I)\ar[ru]\ar[rd]&&X_+\\&&Y\ar[ru]_{p_+}&}
\end{equation}
where the unlabelled maps are natural projections of rational homogeneous bundles over $X_+$.
\end{lemma}

\begin{proof}
We may consider $\cU_-$ as a family of rational curves in $Y$, and the $\DA_{r_+-1}$-bundle $\cY^+\to Y$, obtained as a factorization of the $\DA_{r_+}$-bundle $\cY^+\to X_+$. By Lemma \ref{lem:subtag}, numbering the nodes of $\DA_{r_+-1}$ from $2$ to $r_+$,  the set of indices for which the tag of $\cY^+\to Y$ on the curves of $\cU_-$ is zero is equal to $I$. Then Lemma \ref{lem:map_sections} tells us that there exists a map $f:\cU_-\to \cY^+(I)$ commuting with the projections onto $Y$. On the other hand we have a projection $\cY^+(I)\to\cY^+(N_+)$; since $f$ sends a curve of $\cM_-$ to the minimal section of $\cY^+(I)\to X_+$ over it, and $s_{0_-}$ sends a curve of $\cM_-$ to the minimal section of $\cY^+(N_+)\to X_+$ over it, it follows that the map $f$ makes the diagram (\ref{diag:Y_I}) commutative. 
\end{proof}


\section{Proof of the Main Theorem}\label{sec:proof_main}

In this section we will prove Theorem \ref{thm:main}. We will start by introducing some preliminary results.


\subsection{Some preliminary results}\label{ssec:prelim}

We start with two lemmas regarding projective bundles.

\begin{lemma}\label{lem:diag} Let $X$ be a smooth projective variety of Picard number one. If $p:Y \to X$ is a diagonalizable  $\P^r$-bundle, then  $Y$ has a second fiber type contraction $q:Y \to Z$ if and only if $p$ is trivial. 
\end{lemma}

\begin{proof}
Let $L$ be the ample generator of  $\Pic(X)$,  and denote by $\cO_X(k)$ the line bundle $L^{\otimes{k}}$. We can write $Y=\P(\cE)$ with $\cE \simeq \bigoplus_{i=0}^r \cO_X(a_i)$ and $0=a_0 \le a_1 \le \dots \le a_r$.  The vector bundle $\cE$ is nef but not ample, hence its tautological line bundle $\xi$ is nef but not ample, and it is a supporting divisor for $q$.
If $a_r >0$ then we consider the effective divisor $\P(\bigoplus_{i=0}^{r-1} \cO_X(a_i)) \subset   \P(\cE)$, which is linearly equivalent to $\xi \otimes p^*\cO_X(-a_r)$. This divisor is negative on all the curves contracted by $q$, which are thus contained in it, so  $q$ cannot be of fiber type.                
\end{proof}

\begin{lemma}\label{lem:splittangent}
Let $\cE \simeq \bigoplus_{i=1}^s \cO_{\P^r}(a_i)$ be a non trivial bundle of rank $s \ge 2$ on $\P^r$. 
Then every morphism over $\P^r$ from $\P(\cE)$ to $\P(T_{\P^r})$ is constant over $\P^r$.\end{lemma}

\begin{proof} Assume that $\phi:\P(\cE) \to \P(T_{\P^r})$ is such a morphism. Composing with the second projection $p_2:\P(T_{\P^r}) \to {\P^r}^\vee$ we would obtain that $\P(\cE)$ has a fiber type contraction onto a subset of ${\P^r}^\vee$. By Lemma \ref{lem:diag}, the composition $p_2 \circ \phi\colon \P(\cE)\to {\P^r}^\vee$ factors through the natural projection $\P(\cE) \to \P^r$.
\end{proof}

As a consequence of Lemma \ref{lem:diag}, we obtain the following statement, that allows us to reduce the proof of Theorem \ref{thm:main} to the case in which the two projective bundle structures are not diagonalizable:

\begin{proposition}\label{prop:product}
In the above setting, the bundles $\cY^\pm\to X_{\pm}$ are not diagonalizable, unless $Y=\P^{r_+}\times \P^{r_-}$.
\end{proposition} 

\begin{proof}
If, for instance, $\cY^-\to X_{-}$ is diagonalizable, then the existence of the (fiber type) contraction $Y\to X_+$ implies, by Lemma \ref{lem:diag}, that $Y\simeq\P^{r_-}\times X_-$. Moreover, since $X_-$ has Picard number one, the fact that $Y\simeq\P^{r_-}\times X_-$ is a $\P^{r_+}$-bundle over $X_+$ tells us that $Y\simeq\P^{r_-}\times\P^{r_+}$.
\end{proof}

\begin{lemma}\label{lem:lifting0}
Let $p:Y \to X$ be a $\P^r$-bundle. Let $\pi:\cY \to Y$ be a $G/B$-bundle, whose tag is trivial on the lines in the fibers of $p$. Then $p$ lifts to $\cY$, i.e.,  there exists a smooth variety $\ol{X}$ fitting in a commutative diagram
$$ \xymatrix@R=18pt{
\cY \ar[d]_\pi \ar[r]^{\overline p}&\ol{X}\ar[d]\\
Y \ar[r]_p & X}$$
such that $\overline p$ is a $\P^{r}$-bundle.
\end{lemma}

\begin{proof}
Given a fiber $\P^r$ of $p$, since the tag of the flag bundle $\pi^{-1}(\P^r)\to \P^r$ on the lines of $\P^r$ is equal to zero, by Theorem \ref{thm:trivial} we obtain that $\pi^{-1}(\P^r)\simeq G/B\times\P^r$, so $p\circ\pi$ is a $(G/B\times \P^r)$-bundle. Then the projection $p_2:G/B\times \P^r\to G/B$ extends to a contraction $\ol{p}:\cY\to\ol{X}$, that coincides with $p_2$ fiberwise over $X$. 
\end{proof}

\begin{proposition}\label{prop:diag2}
Let $Y$ be a smooth projective variety supporting two mutually uniform projective bundle structures. With the notation of Section \ref{set:main}, assume further that $r_-=1$. Then, either
\begin{enumerate}
\item  $\delta_+=(d, 0,\dots, 0)$ for some $d \ge 0$, or
\item $\delta_+=(d, 0,\dots,0,d)$ for some $d >0$, $Y\to X_+$ is not diagonalizable, $r_+$ is odd, $p_+$ reduces to  $\PSP(r_++1)$ and the associated flag bundle of type $\DC_{(r_++1)/2}$ is uniform with respect to $\cM_-$, with tag $(d,0,\dots,0)$. 
\end{enumerate}
\end{proposition}

\begin{proof}
Assume first that $\rho_Y=2$.
Since $r_-=1$ then $\cU_-=Y$ and $\cM_-=X_-$; in particular
 we have a  map $s_{0-}:Y \to \cY^+(N_+)$ over $X_+$. 
 
The map $s_{0-}$ sends a fiber $\ell$ of $Y\to X_-$ to a minimal section of $\cY^+(N_+)\to X_+$ over $p_+(\ell)$; the fiber $\ell$ is a minimal section of $Y\to X_+$ over $p_+(\ell)$, so $s_{0-}$ sends minimal sections (over curves of the family $\cM_-$) to minimal sections. 

If $\delta_+(1)=0$ we are in the hypothesis of  Lemma \ref{lem:minimals}, and $s_{0-}$ is relatively constant over the curves of  $\cM_-$. Since $\cM_-$ is a dominating family,  $s_{0-}:Y\to \mathcal Y^+(N_+)$ is relatively constant over $X_+$, thus $s_{0-}$ factors via a section $s'_{0-}:X_+\to \cY^+(N_+)$. Applying Proposition \ref{prop:diag}, we obtain that $Y\to X_+$ is diagonalizable, so  $Y \simeq \P^1 \times \P^{r_+}$ by Proposition \ref{prop:product} and the tag $\delta_+$ is trivial.

In this case in which $\delta_+(1)>0$ we have a map  $\cY^+(N_+)\to Y$, for which $s_{0-}$ is a section. If $N_+=\{1\}$ we are in case (1); otherwise the map $\cY^+(N_+)\to Y$ is not an isomorphism, and $s_{0-}$ gives a nesting $(\cY^+,1,N_+\setminus\{1\})$.  
Using Proposition \ref{prop:nest}, we get the case (2). 

Assume now that $\rho_Y >2$.
The families $\cM_-$, $\cM_+$ define a proper prerelation in the sense of \cite[Definition IV.4.6]{kollar}; 
to these prerelations one can associate a proper proalgebraic relation $\textrm{Chain}(\cU_-,\cU_+)$ (see \cite[Theorem IV.4.8]{kollar}); we denote by $\langle \cU_-, \cU_+ \rangle$ the set-theoretic relation associated with  $\textrm{Chain}(\cU_-,\cU_+)$ .
By \cite[Theorem IV.4.6]{kollar} there exists an open subvariety $Y^\circ \subset Y$ and a proper morphism with connected fibers $\pi:Y^\circ \to Z^\circ$ whose fibers are equivalence classes of $\langle \cU_-, \cU_+ \rangle$.
In particular a general fiber $F$ of $\pi$ is smooth and $\rho_F=2$, since $\cM_\pm$ are unsplit families (see, for instance \cite[Proposition 3]{NO2}).
The restrictions of $p_\pm$ to $F$ are $\P^{r_\pm}$-bundles, and we apply the first part of the proof to get that either $\delta_+=(d,0,\dots,0)$, $d\geq 0$, or $\delta_+=(d,0,\dots,0,d)$, with $d>0$. In the second case we have a contraction $\cY^+(N_+)\to Y$, for which the map $s_{0-}:Y\to \cY^+(N_+)$ is a section, and we conclude, as in the case $\rho_Y=2$, by Proposition \ref{prop:nest}.
\end{proof}

\begin{corollary}\label{cor:flag}
 Let $\cY^+$ be a smooth projective variety  which admits a projective bundle structure $\ol{p}:\cY^+ \to \ol{X}$ and a flag bundle structure $\pi_+:\cY^+ \to X_+$ of type $\cD_+$ on a smooth variety $X_+$ of Picard number one such that $\ol{p}$ is uniform with respect to the families of fibers of the elementary contractions of $\cY^+$ factoring $\pi_+$: 
$$
\xymatrix@R=8pt{
&\cY^+\ar[ld]_{\ol{p}}\ar[rd]^{\pi_+}&\\\ol{X}&&X_+}
$$
Then $\cY^+$ is a rational homogeneous variety.
\end{corollary}

\begin{proof} We will prove this by showing that $\cY^+$ is the image of a contraction of a complete flag manifold.
Let $r$ be the dimension of the fibers of $\ol{p}$, and let $\ol{\pi}:\ol{\cY}\to \ol{X}$ be the complete flag bundle, of type $\DA_{r}$, associated to $\ol{p}$. We have that:
$$
\rho_{\ol{\cY}}=\rho_{\cY^+}+r-1.
$$
By construction, the variety $\ol{\cY}$ has $r$  elementary contractions $\ol{\rho}_j:\ol{\cY}\to \ol{\cY}(j^c)$ over $\ol{X}$ which are $\P^1$-bundles. On the other hand, for each elementary contraction $\rho_j:\cY^+\to \cY^+(j^c)$ over $X_+$ we may apply Proposition \ref{prop:diag2} to 
$$
\xymatrix@R=8pt{&\cY^+\ar[ld]_{\ol{p}}\ar[rd]^{\rho_j}&\\
\ol{X}&&\cY^+(j^c)}
$$ in order to obtain the possible tags $\delta_j$ of $\ol{\cY}\to \ol{X}$ on the fibers $\Gamma_j$ of $\rho_j$. 

If the tag $\delta_j$ is equal to $(d_j, 0, \dots, 0)$, with $d_j\geq 0$, then, by Lemma \ref{lem:subtag}, the tag of $\ol{\cY}\to \cY^+$ on $\Gamma_j$ is equal to $0$, and Lemma \ref{lem:lifting0} tells us that $\rho_j$ can be lifted to $\ol{\cY}$. If this holds for every $j$, then the liftings of the $\rho_j$'s are $\rho_{\cY^+}-1$ $\P^1$-bundles structures on $\ol{\cY}$, whose associated rays $R_j$ in the Mori cone of $\ol{\cY}$ are independent. By construction, the intersection of the kernel $\Nu(\ol{\cY}|\ol{X})$ of the induced map $\Nu(\ol{\cY})\to\Nu(\ol{X})$ with the space generated by the rays $R_j$ is equal to zero. It follows that the $(\rho_{\cY^+}+r-1)$ $\P^1$-bundle structures that we found are given by linearly independent classes in $\Nu(\ol{\cY})$. We then conclude that $\ol{\cY}$ is a complete flag manifold by \cite[Theorem A.1]{OSW}. 

If  for some $i$ we have $\delta_i=(d_i>0, 0, \dots, 0,d_i)$, then, by Proposition \ref{prop:diag2}, $\ol{p}$ reduces to a $\PSP(r+1)$-bundle, and \cite[Proposition 4.11]{MOS7} tells us that all the other $\delta_j$ must be symmetric, hence of the form $(d_j, 0, \dots, 0,d_j)$, with $d_j\geq 0$. In particular
 $\cY^+$ is a quotient of a flag bundle $\ol{\cY}_{\DC}$ of type $\DC_{k}$, with $r=2k-1$. Moreover, Proposition \ref{prop:diag2} tells us also that the tag of $\ol{\cY}_{\DC}\to \ol{X}$ on $\Gamma_j$ equals $(d_j,0,\dots,0)$, for every $j$. Then, by Lemma \ref{lem:subtag}, the tag of $\ol{\cY}_{\DC}\to \cY^+$ (which is a $\DC_{k-1}$-bundle) on the fibers of $p_-$ is $(0,\dots,0)$, and we conclude as in the previous case, by Lemma \ref{lem:lifting0} and
\cite[Theorem A.1]{OSW}.
\end{proof}

 
\subsection{Proof of Theorem \ref{thm:main}}\label{ssec:proof}
We will show that $Y$ is the target of a contraction of a complete flag manifold.
The case $\delta_+=0$ easily follows from Theorem \ref{thm:trivial}; let us then assume that $\delta_+$ is not zero. We introduce the following notation:
given  a point $x\in X_+$, let $P_x\simeq \P^{r_+}$ denote its inverse image in $Y$. We denote by $\cU_x:=\cU_-\times_Y P_x$, and by $\cY^x:=\cY^+\times_Y P_x$ the restrictions of $\cU_-$ and $\cY^+$ to $P_x$. 
Note that $\cY^x$ is the complete flag of $\P^{r_+}$, $\cY^x\simeq\DA_{r_+}(1,\dots,r_+)$. 

\medskip

\noindent{\em Step $1$: If  $\,\cU_x \to P_x$ is diagonalizable for some $x \in X^+$, then either $\delta_+=(d >0, 0, \dots, 0)$, or $r_+$ is odd, $\delta_+=(d >0, 0, \dots, 0, d)$, and $Y$ is a quotient of a flag bundle $\cY^+_{\DC}$ of type $\DC_{(r_++1)/2}$, uniform with respect to $\cM_-$, with tag $(d>0,0,\dots,0)$.}

\medskip

Assume first that $\delta_+(1)=0$ and set $I:=N_+\cup\{1\}$. By Lemma \ref{lem:YI} there exists a morphism $f: \cU_- \to \cY^+(I)$ fitting in a commutative diagram
$$
\xymatrix@R=8pt{\cU_-\ar[rr]^{f}\ar[rd]_{\ev_-}&&\cY^+(I)\ar[ld]\\&Y&}
$$
We now consider the restriction of $f$ to $\cU_x$, which gives a morphism $\cU_x\to \cY^x(I)$, that we denote also by $f$. Since, by hypothesis, $\cU_x\to P_x$ is diagonalizable, then, by Proposition \ref{prop:sect_diag}, it has a section $s:P_x\to \cU_x$; composing it with $f:\cU_x\to \cY^x(I)$, we get a section of the contraction $\cY^x(I)\to P_x$. Since $\delta_+\neq 0$, then $N_+\neq\emptyset$, 
so $\cY^x_I\to P_x$ is not an isomorphism, and the section $f\circ s:P_x\to \cY^x(I)$ is a nesting. Then, by Proposition \ref{prop:nest}, $I=\{1,r_+\}$, $\cY^x(I)=\P(T_{P_x})$,  
and Lemma \ref{lem:splittangent} tells us that $f:\cU_x\to \cY^x(I)$ factors via the projection $\cU_x\to P_x$. In particular $f:\cU_-\to \cY^+(I)$ contracts the fibers of $\ev_-:\cU_-\to Y$, and so it factors through it, and we obtain a section $\sigma:Y\to \cY^+(I)$ of the projection $\cY^+(I)\to Y$. We may then apply Proposition \ref{prop:nest} to claim that the tag $\delta_+$ is symmetric, contradicting that $I=\{1,r_+\}$. We have thus shown that $\delta_+(1) \not = 0$. 

Assume that there exists another nonzero element in the tag, i.e. that $N_+\neq \{1\}$, 
and consider the map $s_{0-}$ provided by Lemma \ref{lem:map_sections}, fitting in the diagram
$$
\xymatrix@R=8pt{\cU_-\ar[rr]^{s_{0-}}\ar[rd]_{\ev_-}&&\cY^+(N_+)\ar[ld]\\&Y&}
$$
Restricting the diagram to $P_x$ and arguing as in the case $\delta_+(1)= 0$, we get a section of $\cY^+(N_+) \times_YP_x \to P_x$. Arguing as in the previous case, we get a section $Y\to \cY^+(N_+)$, hence a nesting $(\cY^+,1,r_+)$, and the statement follows again by Proposition \ref{prop:nest}.
\medskip

\noindent{\em Step $2$: If  $\,\cU_x \to P_x$ is diagonalizable for some $x$, then $Y$ is rational homogeneous.}

\medskip

By the previous step, replacing $\cY^+$ with $\cY^+_{\DC}$ if necessary, we may assume that $\delta_+=(d, 0, \dots, 0)$ for some integer $d>0$. In particular the tag of $\cY^+\to Y$ over the curves of $\cM^-$ is, by  Lemma \ref{lem:subtag}, equal to zero. In particular the projection $p_-:Y\to X_-$ lifts to $\cY^+$ by Lemma \ref{lem:lifting0}:
$$
\xymatrix@R=20pt{\ol{X}\ar[d]&\cY^+\ar[l]_{\ol{p}_-}\ar[rd]\ar[d]&\\X_-&Y\ar[l]^{p_-}\ar[r]_{p_+}&X_+}
$$

We claim that $\ol{p}_-$ is uniform with respect to the families of fibers of the elementary contractions $\rho_j$ of $\cY^+$ over $X_+$.  By construction the fibers of $\cY^+\to X_-$ are isomorphic to $\P^{r-}\times F$, where $F$ denotes a fiber of $\cY^+\to Y$, hence $\ol{p}_-$ is trivial on the fibers of $\rho_j$ for $j=2, \dots, r_+$. On the other hand the tag of $\ol{p}_-$ on the fibers of $\rho_1$ is the tag of $p_-$ on the images of those curve in $X_-$, and the claim follows from the fact that $p_-$ is uniform with respect to $\cM_+$ by assumption.
Now the homogeneity of $\cY^+$ follows by Corollary \ref{cor:flag}.

\medskip

\noindent{\em Step $3$: Case $r_- \neq r_+$.}

\medskip

Without loss of generality, assume that $r_- < r_+$. For every $x$, $\cU_x$ is a $\P^{r_--1}$-bundle on $P_x\simeq\P^{r_+}$, uniform with respect to lines, hence  classical results on uniform bundles (cf. \cite{VdV,Sa}) imply that $\cU_x$ is diagonalizable, and we conclude by Step $2$. 

\medskip

\noindent{\em Step $4$: Case $r_- = r_+=:r$.}

\medskip
 
By Step $2$, we are left with the case in which $\cU_x\to P_x$ is non diagonalizable for every $x\in X_+$. In particular, following \cite{EHS}, all these restrictions are isomorphic either to $\P(T_{P_x})$ or to $\P(\Omega_{P_x})$.
We observe that, looking at the composition:

$$
\xymatrix{\cY^-\ar[r]\ar@/^1pc/[rr]& \cU_-\ar[r]& Y}
$$
the bundle $\cY^-\to Y$ is the complete flag bundle of type $\DA_{r-1}$ associated to the $\P^{r-1}$-bundle $\cU_-\to Y$. In particular, $\cY^-\times_YP_x\to P_x$ is the complete flag bundle of type $\DA_{r-1}$ determined by $\P(T_{P_x})$ or by $\P(\Omega_{P_x})$ over $P_x$, that is the complete flag manifold of the projective space $P_x$.

Since this holds for every point $x\in X_+$, we obtain that the fibers of $\cY^-\to X^+$ are isomorphic to $\DA_r(1,\dots,r)$, and we may say that $\cY^-$ is a flag bundle of type $\DA_r$ over $X_+$. In particular, it has $r$ elementary contractions over $X_+$ that are $\P^1$-bundles, and the corresponding rays of the Mori cone $\cNE{\cY^-}$ generate $\Nu(\cY^-|X_+)$. On the other hand, by definition, $\cY^-$ has $r$ elementary contractions over $X_-$ that are $\P^1$-bundles, whose rays generate $\Nu(\cY^-|X_-)$. Since the contractions $\cY^-\to X_-$, and $\cY^-\to X_+$ are different, then at least one of the rays of the elementary contractions of $\cY$ over $X_+$ is not contained in $\Nu(\cY^-|X_-)$. We then deduce that $\cY^-$ has at least $r+1$ independent elementary contractions that are $\P^1$-bundles and, noting that $\rho_{\cY^-}=r+1$ and using \cite[Theorem A.1]{OSW}, we conclude that $\cY^-$ is a complete flag manifold. From this it follows that also $Y$, as target of one of its contractions, is rational homogeneous.


\section{Applications}\label{sec:drums}

In this section we will use Theorem \ref{thm:main} to prove that some varieties admitting a $\C^*$-action are rational homogeneous. We first recall some preliminaries on $\C^*$-actions, see \cite[$\S$2.3]{ORSW}, \cite[$\S$1.B, $\S$1.C]{RW} for further details. Let $(Z,L)$ be a polarized pair, i.e., $Z$ is a  smooth complex projective variety and $L$ is an ample line bundle on $Z$. Assuming that such a polarized pair admits a non trivial $\C^*$-action, we may associate to every fixed point component an integer. To this end, we take a linearization $\mu\colon \C^*\times L\to L$ so that for every fixed point component $X$ we get $\mu(X)\in \Hom{(\C^*, \C^*)}\simeq \Z$. In this way, denoting by $\cX$  the set of the irreducible fixed point components, we obtain a map $\cX\to \Z$, which by abuse we continue to denote by $\mu$.

\begin{definition}\label{def:bandwidth} Let $(Z,L)$ be a polarized pair with a $\C^*$-action and a linearization $\mu$ on $L$. The {\em bandwidth} of the action  is defined as $|\mu|:=\mu_{\max}-\mu_{\min}$, where $\mu_{\max}$ and $\mu_{\min}$ denote the maximal and minimal value of the function $\mu$.  
 \end{definition}
 
We will call {\em sink} the unique fixed point component $X_{-}$ such that $\mu{(X_-)}=\mu_{\min}$, and {\em source} the unique fixed point component $X_{+}$ such that $\mu{(X_+)}=\mu_{\max}$. Varieties with small bandwidth have been studied in \cite{RW} applying adjunction theory,  and in \cite{ORSW} using tools from birational and projective geometry. In what follows, we focus on bandwidth one varieties (see \cite[$\S$4]{ORSW}) and we use their relation with  special varieties, called drums, which we now define:

\begin{definition}\label{def:drum} Let $Y$ be a normal projective variety with $\rho_Y=2$ and two elementary contractions: 
$$
\xymatrix@R=8pt{&Y\ar[dl]_{p_{-}}\ar[dr]^{p_{+}}\\X_{-}&&X_{+}}
$$
Let $L_{\pm}\in\Pic(Y)$ be the pullbacks via $p_{\pm}$ of some ample line bundles in $X_\pm$. Then the vector bundle $\cE:=L_{-}\oplus L_+$ is semiample and there exists a contraction $\varphi$ of $\P(\cE)$, with supporting line bundle $\cO_{\P(\cE)}(1)$. The image $\varphi(\P(\cE))$ will be called the {\em drum} associated to the triple $(Y,L_{-},L_{+})$.
\end{definition}

\begin{example}\label{ex:horo_drum}
In the case in which $Y$ is rational homogeneous (so that $Y$ is one of the varieties listed in Remark \ref{rem:list}), the drum constructed upon it is a smooth horospherical variety of Picard number $1$, whose classification can be found in \cite{Pas}. In fact every rational homogeneous variety $\cD(i,j)$ listed in Remark \ref{rem:list} corresponds to a triple $(\cD,\alpha_i,\alpha_j)$ in the list of \cite[Theorem~1.7]{Pas}. Following Pasquier (\cite[Section 1.3]{Pas}), each of these triples determines a horospherical variety, constructed as the closure $Z$ of the $G$-orbit of the point $[v_i+v_j]\in\P((V_{i}\oplus V_{j})^\vee)$; 
here, for $k=i,j$, $V_k$ denotes the irreducible representation associated to the fundamental weight corresponding to the $k$-th node of $\cD$, and $v_k$ the corresponding highest weight vector. Writing $G[v_i+v_j]$ as $G/K$ 
for a certain subgroup $K\subset G$, the variety $G/K$ is a $\C^*$-bundle over $\cD(i,j)=G/P(D\setminus\{i,j\})$, and the torus $P(D\setminus\{i,j\})/K\simeq \C^*$ acts naturally on $Z$. Then, denoting by $\cO_Z(1)$ the tautological line bundle of the embedding $Z\subset \P((V_{i}\oplus V_{j})^\vee)$, one may show that, after quotienting $\C^*$ by a finite subgroup, the action of $\C^*$ on $(Z,\cO_Z(1))$ has bandwidth $1$, with fixed components $Z\cap \P(V_k^\vee)\simeq \cD(k)$, $k=i,j$.
\end{example}

We conclude the paper by proving that Campana--Peternell Conjecture holds for bandwidth one varieties of Picard number one. 

\begin{theorem}\label{thm:drumRH}
Let $(Z,L)$ be a polarized pair with a $\C^*$-action of bandwitdh one. Assume that $\rho_Z=1$ and that the tangent bundle $T_Z$ is nef. Then $Z$ is a rational homogeneous manifold.
\end{theorem}

\begin{proof} Let us denote by $X_-$ and $X_+$ the sink and the source of the action and by $\cN_-$ \and $\cN_+$ their normal bundles in $Z$. Let $\alpha:Z^\flat \to Z$ be the blowup of $Z$ along $X_- \sqcup X_+$, with exceptional divisors  $Y_-=\P(\cN_-^\vee)$ and $Y_+=\P(\cN_+^\vee)$.
By \cite[Theorem 4.6]{ORSW} and its proof, it follows that $Z$ is a drum, constructed upon the projections of $Y:=Y_- \simeq Y_+$ onto $X_\pm$. 

If $\dim{X_-}=0$, necessarily $X_+$ is a divisor by the Bend-and-Break lemma. Then we obtain that $Y_+\simeq X_+\simeq\P^{\dim Z-1}$, and so $Z\simeq \P^{\dim Z}$, because it is a variety of Picard number one containing a projective space as an ample divisor.  
On the other hand, if $\codim(X_+,Z)=1$, then $Y_-$ is isomorphic to $X_+$ and, applying \cite[Lemma~2.9]{ORSW}, it has Picard number one; this is only possible if $X_-$ is a point and, subsequently, $Z\simeq \P^{\dim Z}$. Similar arguments hold interchanging sink and source so, summing up, we may assume that $1\leq\dim X_{\pm} \leq\dim Z-2$.

Arguing as in the proof of \cite[Theorem 4.6]{ORSW}, the projections $p_\pm:Y \to X_\pm$ are different projective bundle structures. By \cite[Lemma 2.9]{ORSW} we have $\rho_{X_\pm}=1$, hence $\rho_Y=2$. Moreover, since $Z$ is a drum constructed upon $Y$, $Z^\flat$ is the projectivization over $Y$ of a decomposable rank two bundle, with two sections $s_\pm:Y \to Z^\flat$ whose images are $Y_\pm$. 
 
Setting $L_\pm:=s_\pm^*(\alpha^*L|_{Y_\pm})$ we can write
$Z^\flat=\P_Y(L_-\oplus L_+)$ with  natural projection $\pi\colon Z^\flat\to Y$; by construction the line bundles $L_\pm$ are nef and,  denoting by $\ell_\pm$ a line in a fiber of $p_\pm$, we have $L_- \cdot \ell_- =L_+ \cdot \ell_+=0$. Recalling that 
\begin{equation}\label{eq:Y_+}
Y_+=\alpha^*L-\pi^*L_- \qquad Y_-=\alpha^*L-\pi^*L_+
\end{equation} 
we have $s_+^*(Y_+|_{Y_+})=L_+-L_-$ and $s_-^*(Y_-|_{Y_-})=L_--L_+$. Intersecting with $\ell_\pm$ we get $L_- \cdot \ell_+=L_+ \cdot \ell_-=1$. 
By Equation (\ref{eq:Y_+}) we also see that the line bundles $M_\pm:=\alpha^*L -Y_\pm$ are nef; notice that $M_\pm$ are the tautological line bundles of the  projectivization of the vector bundles $p_{\pm}^*\cN^\vee_\pm \otimes L_\pm$, which are then nef.
On the other hand, by the assumption on $T_Z$, we also have that the normal bundles $\cN_\pm$ are nef.
The two conditions together imply that $p_\pm$ are mutually uniform. In fact, if $\ell_-$ is a line in a fiber of $p_-$, then we can write
$$(p_+^*\cN_+^\vee \otimes L_-)|_{\ell_-} \simeq \sum_{i=1}^{\rk \cN_+} \cO_{\P^1}(a_i)\qquad a_i \ge 0,$$ 
and we obtain that 
$$(p_+^*\cN_+)|_{\ell_-} \simeq \sum_{i=1}^{\rk \cN_+} \cO_{\P^1}(1- a_i),$$
which implies that the possible values of the $a_i$'s are zero and one. Then the number of $\cO(1)$ summands
equals  $\deg(\cN_+|_{\ell_-})$. The same argument can be repeated for $\cN_-$. 

Using Theorem \ref{thm:main} we deduce that $Y$ is one of the rational homogeneous manifolds $\cD(i,j)$ listed in Remark \ref{rem:list}. In particular $Z$ is a horospherical variety (see Example \ref{ex:horo_drum}), and the result follows from \cite[Theorem 1.2]{Li}.
\end{proof}

\end{document}